\documentclass[12pt]{article}
\usepackage{amsmath,amsthm,amssymb,amsfonts,mathrsfs}
\usepackage[hypertex, bookmarks, colorlinks=true, plainpages = false,
citecolor = green, urlcolor = blue, filecolor = blue]{hyperref}
\newtheorem{theorem}{Theorem}[section]
\newtheorem{lemma}[theorem]{Lemma}
\newtheorem{proposition}[theorem]{Proposition}
\newtheorem{corollary}[theorem]{Corollary}
\newtheorem{definition}[theorem]{Definition}
\newtheorem{example}[theorem]{Example}

\newtheorem{remark}[theorem]{Remark}

=23cm\headheight=0cm\topskip=0cm\topmargin=0cm
\def\a{\bar{a}}\def\b{\bar{b}}\def\c{\bar{c}}
\def\x{\bar{x}}\def\y{\bar{y}}
\def\z{\bar{z}}\def\r{\bar{r}}
\def\e{\bar{e}}

\def\Nn{\mathbb N}

\def\Rn{\mathbb R}

\def\0{\sf 0}

\begin{document}

\begin{center}
{\Large\sc Extreme types and extremal models}
\bigskip

{{\bf Seyed-Mohammad Bagheri}}
\vspace{3mm}

{\footnotesize {\footnotesize Department of Pure Mathematics, Faculty of Mathematical Sciences,\\
Tarbiat-Modares University, Tehran, Iran, P.O. Box 14115-134}\\
e-mail: bagheri@modares.ac.ir}
\end{center}
\bigskip

\begin{abstract}
In the affine fragment of continuous logic, type spaces are compact convex sets.
I study some model theoretic properties of extreme types.
It is proved that every complete theory $T$ has an extremal model, i.e. a model which realizes only extreme types.
Extremal models form an elementary class in the full continuous logic sense if and only if
the set of extreme $n$-types is closed in $S_n(T)$ for each $n$.
Also, some applications are given in the special cases where the theory has a compact or first order model.
\end{abstract}
\vspace{7mm}


\bigskip
\section{Introduction}
First order omitting types theorem (or its continuous logic variant) is usually proved by Henkin's method.
The theorem holds only for countable languages and its proof uses the connective $\wedge$ in an essential way.
The goal of the present paper is mainly to prove omitting types theorem in the framework of linear continuous logic.
This is the affine fragment of continuous logic based on the additive structure of the real numbers
as value space. The conjunction operator is not allowed here.
So, the expressive power is reduced and sets of conditions (types) are more easily realized.
In this fragment, complete types correspond to normalized positive linear functionals on the normed
vector space of formulas (while, they correspond to characters in full continuous logic \cite{Bagheri-Lip}).
In particular, type spaces form compact convex sets (while compact sets in the full variant).
This additional structure imposes some new features on the linear variant of continuous model theory.

A fundamental notion in the study of compact convex sets is extremity.
A more special notion is exposedness.
We show in this paper that every complete theory has a model which omits all non-extreme types.
Such models will be called \emph{extremal}. The class of extremal models forms an abstract elementary class.
It forms an elementary class in the framework of full continuous logic if and only if
the set of extreme $n$-types is compact in the logic topology for all $n$.
In this case, extreme types coincide with the types (in the continuous logic sense) of the extremal theory.
So, some extreme types may be omitted again in extremal models using
the full continuous logic variant of omitting types theorem.

A remarkable consequence of disallowing conjunction and disjunction operators in the formation of formulas is that (linearly)
complete theories may have both compact and non-compact models.
If the theory has a compact model, extreme types are exactly those which are realized in all compact models.
In general, exposedness (a linear variant of principality) may be a candidate for non-omittability.

Omitting non-extreme types is interesting in its own right. Countability of language is not needed.
All non-extreme types can be omitted simultaneously.
Moreover, the proof can be adopted to find arbitrarily large models omitting all non-extreme types
if the theory allows an infinite sequence of pairwise distant points (e.g. if the theory has an infinite first order model).
An advantage of working with extreme types (comparing with exposed ones)
is that a wide range of literature concerning them exists in convex analysis.
Some mathematical properties of extreme types then find interesting applications in linear model theory.
The proofs given in this paper work for complete types in complete theories.
More general cases (incomplete theories or types) need further study.

In the next section, we introduce linear continuous logic. More details can be found in \cite{Bagheri-Lip}.
In the third section we study facial types, i.e. partial types which correspond to
faces of the type spaces $S_n(T)$. Complete facial types are called extreme types.
We show in section four that every complete theory has an extremal model.
In the last section, we give some applications of the extremal omitting types theorem.
In particular, we consider situations where the extremal models form an elementary class
in the full continuous logic sense. We also show that linearly complete theories can not be separably categorical.

\section{Linear continuous logic}
All metric spaces are assumed to be complete with bounded diameter.
Let $X,Y$ be metric spaces. A function $f:X\rightarrow Y$
is \emph{$\lambda$-Lipschitz} if
$$\ \ \ \ \ d(f(x),f(y))\leqslant\lambda d(x,y) \hspace{14mm} \forall x,y\in X.$$
Full continuous logic is usually presented as a direct generalization of first order logic.
So, the value space is taken to be the unit interval and Boolean connectives $\wedge$ and $\neg$ are used.
To achieve the linear fragment of continuous logic, we shift to the whole real line and use its
algebraic operations as connectives.
A Lipshcitz language is a set $L$ consisting of constant symbols as well as function
and relation symbols of various arities. To each function symbol $F$
is assigned a Lipschitz constant $\lambda_F\geqslant0$ and to each relation symbol $R$ is assigned
a Lipcshitz constant $\lambda_R\geqslant0$.
It is always assumed that $L$ contains a distinguished binary relation symbol $d$ for metric.
An $L$-structure is a complete metric space $(M,d)$ on which the symbols of $L$ are
appropriately interpreted.
So, for constant symbol $c\in L$ one has that $c^M\in M$, for $n$-ary function symbol $F\in L$ the function $F^M:M^n\rightarrow M$
is $\lambda_F$-Lipschitz and similarly for $n$-ary relation symbol $R\in L$, the function $R^M:M^n\rightarrow\Rn$
is $\lambda_R$-Lipschitz with $\|R^M\|_{\infty}\leqslant1$.
In particular, $d(x,y)\leqslant1$ for all $x,y$.
Here, we put the metric $\sum_{i<\alpha}2^{-i}d(x_i,y_i)$ on $M^\alpha$ for every $1\leqslant\alpha\leqslant\omega$.

The set of $L$-terms is defined as in first order logic.
The set of $L$-formulas of linear continuous logic is inductively defined as follows:
$$r,\ \ \ d(t_1,t_2),\ \ \ R(t_1,...,t_n),\ \ \ r\phi,\ \ \ \phi+\psi,\ \ \ \sup_x\phi,\ \ \ \inf_x\phi$$
where  $r\in\Rn$, $R\in L$ is $n$-ary relation symbol and $t_1,...t_n$ are $L$-terms.
In full continuous logic, one further allows $\phi\wedge\psi$ and $\phi\vee\psi$ to be formulas.
So, the expressive power of the linear fragment is significantly reduced with respect to full continuous logic.

A formula without free variable is called a \emph{sentence}.
If $\phi(\x)$ is a formula, $M$ is a structure and $\a\in M$,
the real value $\phi^M(\a)$ is defined by induction on the complexity of $\phi$.
Every map $\phi^M:M^n\rightarrow\Rn$ is then Lipschitz with respect to some $\lambda_{\phi}$
and bounded by some $\mathbf b_\phi$ (both depending only on $\phi$).
Let $M,N$ be $L$-structures. We write $M\equiv N$ if $\sigma^M=\sigma^N$
for every sentence $\sigma$ in $L$.
Similarly, we write $M\preccurlyeq N$ if $\phi^M(\a)=\phi^N(\a)$ for every $\a\in M$ and formula $\phi$ in $L$.
These notions are weaker than the corresponding notions in full continuous logic.
Below, to distinguish between the two notions, the full variants
will be denoted by $\equiv_{\mathrm{CL}}$, $\preccurlyeq_{\mathrm{CL}}$ etc.
(CL abbreviates Continuous Logic).

Expressions of the form $\phi\leqslant\psi$ are called \emph{conditions}
(resp. closed conditions if $\phi,\psi$ are sentences).
$\phi=\psi$ is an abbreviation for $\{\phi\leqslant\psi, \psi\leqslant\phi\}$.
A set of closed conditions is called a \emph{theory}. More precisely, it must be called a linear (or affine) theory.
$M$ is model of $\phi\leqslant\psi$ if $\phi^M\leqslant\psi^M$.
The notions $M\vDash T$ and $T\vDash\phi\leqslant\psi$ are defined in the obvious way.
A theory $T$ is \emph{linearly satisfiable} if for every conditions $\phi_1\leqslant\psi_1,\ ...,\ \phi_n\leqslant\psi_n$ in $T$
and $0\leqslant r_1,...,r_n$, the condition $\sum_ir_i\phi_i\leqslant\sum_ir_i\psi_i$ is satisfiable.
The linear compactness theorem can be proved by the ultramean construction \cite{Bagheri-Lip}
or by Henkin's method \cite{Bagheri-Safari}.

\begin{theorem} \label{compactness}
(Linear compactness) Every linearly satisfiable theory is satisfiable.
\end{theorem}

In particular, for a sentence $\sigma$, if $\sigma\leqslant0$ and $0\leqslant\sigma$ are
satisfiable, then $\sigma=0$ is satisfiable. An easy consequence of linear compactness is
that every model of cardinality at least 2 has arbitrarily large elementary extensions.
So, compact models may be elementarily equivalent to non-compact models.
Usual results such as elementary amalgamation, L\"{o}wenheim-Skolem
and elementary chain theorems hold as in full continuous logic.
A satisfiable theory $T$ is \emph{complete} if for each sentence $\sigma$
there is a unique $r$ such that $\sigma=r\in T$.
If $T$ is complete, either the singleton structure is its unique model or
$T$ has models of arbitrarily large cardinalities. In the first case, we say $T$ is \emph{trivial}.

Let $T$ be a satisfiable theory. Let $|\x|=n$. An \emph{$n$-type} for $T$ is a maximal set $p(\x)$ of
conditions $\phi(\x)\leqslant\psi(\x)$ such that $T\cup p(\x)$ is satisfiable.
We can explain this in a more systematic way if $T$ is complete.
Let $\mathbb D_n(T)$ (or $\mathbb D_n$ for short) be the vector space of formulas with free variables $\x$
where $\phi(\x)$, $\psi(\x)$ are identified if $\phi^M=\psi^M$ for some (and hence every) model $M\vDash T$.
In other words, $$\mathbb D_n=\{\phi^M:\ \phi(\x)\ \mbox{is\ an}\ L-\mbox{formula}\}\subseteq \mathbf{C}_b(M^n).$$
$\mathbb D_n$ is a partially ordered vector space.
If $p$ is an $n$-type, then for each $\phi(\x)$ there is a unique real number $p(\phi)$
such that $\phi=p(\phi)$ belongs to $p$.
In fact, the map $\phi\mapsto p(\phi)$ is linear and depends only to $\phi^M$.
So, thanks to linear compactness theorem, an $n$-type can be redefined as a linear
functional $p:\mathbb D_n(T)\rightarrow\Rn$ such that for every formula
$\phi(\x)$, the condition $\phi=p(\phi)$ is satisfiable in some model of $T$.
In particular, the evaluation map $tp(\a):\phi(\x)\mapsto\phi^M(\a)$ is a type called the type of $\a$.
Types of this form are called \emph{realized} types. The set of all $n$-types of $T$ is denoted by $S_n(T)$.

$\mathbb D_n$ is also equipped with the sup-norm.
So, $S_n(T)$ is exactly the state space of $\mathbb D_n$, i.e. the set of
positive linear functionals $f:\mathbb D_n\rightarrow\Rn$ such that $f(1)=1$.
In particular, if $p$ is such a functional and $p(\phi)=r$, then $r\leqslant\phi$ is satisfiable with $T$.
Otherwise, for some $\epsilon>0$, we must have that $T\vDash\phi\leqslant r-\epsilon$
which implies that $p(\phi)\leqslant r-\epsilon$.
Similarly, $\phi\leqslant r$ and hence $\phi=r$ are satisfiable with $T$.
Note that $S_n(T)$ is a convex set, i.e. if $p,q$ are $n$-types and $0\leqslant\lambda\leqslant1$,
then $\lambda p+(1-\lambda)q$ is an $n$-type.
Types over parameters from $A\subseteq M\vDash T$ are defined similarly.
$S_n(A)$ denotes the set of $n$-types over $A$.
By Banach-Alaoglu theorem, $S_n(T)$ (as well as $S_n(A)$)
is a compact subset of the unit ball of $\mathbb D_n^*$ (similarly for $S_n(A)$).
This topology is what logicians call \emph{logic topology}.
Note that the map $\a\mapsto tp(\a)$ is continuous.

\begin{definition}
\em{ A model $M\vDash T$ is called \emph{$\kappa$-saturated} if for each $A\subseteq M$ with $|A|<\kappa$,
every $p\in S_n(A)$ is realized in $M$.}
\end{definition}

$\kappa$-saturation in the full continuous logic sense has nothing to do with $\kappa$-saturation
in the present sense (e.g. compact structures are not saturated in the present sense).
The elementary diagram of a model $M$ is defined in the usual way, i.e. $$ediag(M)=\{\phi(\a)=0:\ \phi^M(\a)=0\}.$$
It is then proved (assuming $M\subseteq N$) that $M\preccurlyeq N$ if and only if $N\vDash ediag(M)$.

\begin{lemma}
Let $A\subseteq M$ and $p(\x)$ be a type over $A\subseteq M$.
Then $p$ is realized in some $M\preccurlyeq N$.
\end{lemma}
\begin{proof} It is sufficient to show that $ediag(M)\cup p(\x)$ is linearly satisfiable.
Indeed, for each condition $0\leqslant\phi(\a,\b)$ satisfied in $M$, where $\a\in A$ and $\b\in M-A$,
the condition $0\leqslant\sup_{\y}\phi(\a,\y)$ is satisfiable with $p$ by definition.
\end{proof}
\bigskip

By usual chain arguments one shows that for every $M$ and $\kappa$,
there exists $M\preccurlyeq N$ which is $\kappa$-saturated.
Let $M\vDash T$ be an $\aleph_0$-saturated. The metric topology on $S_n(T)$ is defined by
$$d(p,q)=\inf\{d(\a,\b):\ \ \a\vDash p,\ \b\vDash q,\ \a,\b\in M\}.$$
The metric topology is complete and generally finer than the logic topology.
For every model $K\vDash T$, the map $$K^n\rightarrow S_n(T)\ \ \ \ \a\mapsto tp(\a)$$
is $1$-Lischitz.

It is also not hard to show that if $M$ is $\aleph_0$-saturated, then for every formula $\phi(\x)$
the range of $\phi^M$ is a closed interval.

Linear variant of the ultraproduct construction is the ultramean construction.
We recall that a probability charge space is a triple $(I,\mathscr A,\mu)$
where $\mathscr A$ is a Boolean algebra of subsets of $I$ and
$\mu:\mathscr A\rightarrow[0,1]$ is a finitely additive probability measure.
If $\mathscr A=P(I)$, $\mu$ is called an \emph{ultracharge}.
Let $(M_{i}, d_{i})_{i\in I}$ be a family of $L$-structures and $\mu$ be an ultracharge on $I$.
First define a pseudo-metric on $\prod_{i\in I}M_i$ by setting
$$d(a,b)=\int d_{i}(a_i,b_i)d\mu.$$
$d(a,b)=0$ is an equivalence relation. The equivalence class of $a=(a_i)$ is denoted by $[a_i]$.
Let $M$ be the set of equivalence classes.
Then, $d$ induces a metric on $M$ which is denoted again by $d$.
So, $d([a_i],[b_i])=\int d_{i}(a_i,b_i)d\mu$.
Define an $L$-structure on $(M,d)$ as follows:
$$c^M=[c^{M_i}]$$
$$F^M([a_i],...)=[F^{M_i}(a_i,...)]$$
$$R^M([a_i],...)=\int R^{M_i}(a_i,...)d\mu.$$
where $c,F,R\in L$.
One verifies that $F^M$ and $R^M$ are well-defined and $M$ is a (possibly incomplete)
$L$-structure in its own right (which may be completed by usual arguments).

The structure $M$ is called the ultramean of the structures $M_i$ and is denoted by $\prod_\mu M_i$.
An ultrafilter $\mathcal F$ corresponds to the $0-1$ valued ultracharge
$\mu$ where $\mu(A)=1$ if $A\in\mathcal F$ and $\mu(A)=0$ otherwise.
In this case, $\prod_\mu M_i$ is the ultraproduct $\prod_{\mathcal F}M_i$
and by {\L}o\'{s} theorem, for every $\sigma$ in the full continuous logic sense
one has that $\sigma^M=\lim_{i,\mathcal F}\sigma^{M_i}$.
In the ultracharge case, we have the following theorem holding only for linear formulas.

\begin{theorem} \label{th1}
For every $\phi(x_{1}, \ldots, x_{n})$ and $[a^{1}_{i}], \ldots, [a^{n}_{i}]\in M$,
$$\phi^{M}([a^{1}_{i}],\ldots, [a^{n}_{i}])=
\int\phi^{M_{i}}(a^{1}_{i},\ldots, a^{n}_{i})d\mu.$$
\end{theorem}

If $M_i=N$ for all $i$, the ultramean structure is denoted by $N^\mu$.
In this case, the map $a\mapsto [a]$ is an elementary embedding of $N$ into $N^\mu$.
The linear variant of Keisler-Shelah theorem is proved in \cite{Isomorphism}.
Taking suitable ultracharges, one obtains arbitrarily large elementary extensions of $M$
if $|M|\geqslant2$.

\section{Facial types}
A subset $K$ of a locally convex Hausdorff topological vector space $X$ is convex if for every
$0\leqslant\lambda\leqslant1$ and $p,q\in K$, one has that $\lambda p+(1-\lambda)q\in K$ .
Let $K$ be compact and convex.
A non-empty convex $F\subseteq K$ is called a face if for every $p,q\in K$ and $0<\lambda<1$,
$\lambda p+(1-\lambda)q\in F$ implies that $p,q\in F$ (one may only use $\lambda=\frac{1}{2}$).
A point $p$ is extreme if $\{p\}$ is a face. Let $T$ be a complete $L$-theory.
Then, $S_n(T)$ is a compact convex set in $\mathbb D_n^*$.
We will see that some interesting properties of $T$ are related to the extreme
points of $S_n(T)$ which we call extreme types.

In first order logic (as well as continuous logic), the set of types realized in a model
$M\vDash T$ is dense in $S_n(T)$. In linear continuous logic, the situation is different.
For example, $S_1(\textrm{PrA})$ (where PrA is the linearly complete theory of probability algebra)
is a $1$-dimensional simplex, hence isometric to the unit interval $[0,1]$ (see \cite{Bagheri-Lip}).
While, $\{0,1\}$ is a model of PrA which realizes only the extreme types.
However, we have a similar result.
First, I recall the Riesz-Markov representation theorem on arbitrary metric spaces.

\begin{theorem} \label{Riesz-Markov noncompact}
Let $(M,d)$ be a metric space and $\overline M$ be its Stone-\v{C}ech compactification.
If $\Lambda:C_b(M)\rightarrow\Rn$ is bounded linear and positive,
then there exists a unique finite Borel measure $\mu$ on $\overline M$
such that
$$\ \ \ \ \Lambda(f)=\int \bar f d\mu\ \ \ \ \ \ \ \ \ \forall f\in C_b(M)$$
where $\bar f$ denotes the extension of $f$ on $\overline{M}$.
\end{theorem}

The set of all Borel probability measures on $M$ is denoted by $\mathcal P(M)$.
Also, $E_n(T)$ denotes the set of extreme types, i.e. extreme points of $S_n(T)$.

\begin{proposition} \label{dense-realization}
Let $M\vDash T$. Then, every $p\in E_n(T)$ is the limit (in the logic topology) of types realized in $M$.
\end{proposition}
\begin{proof}
Consider the case $n=1$. We may assume that $M$ (and hence $\overline{M}$) is separable.
Let $\zeta:\mathcal P(\overline{M})\rightarrow S_1(T)$ be the function defined by
$\zeta(\mu)=p_\mu$ where $p_\mu(\phi)=\int\overline{\phi^M} d\mu$ for all $\phi$.
Note that $\zeta$ is affine, i.e. $\zeta(\lambda\mu+(1-\lambda)\nu)=\lambda\zeta(\mu)+(1-\lambda)\nu$.
Let $p(x)\in S_1(T)$. By Kantorovich extension theorem
(\cite{Aliprantis-Inf} Th. 8.32.), the map defined by $\Lambda(\overline{\phi^M})=p(\phi)$,
for all $\phi$, extends to a positive linear functional on $C(\overline{M})$.
So, by Theorem \ref{Riesz-Markov noncompact}, there exists a probability measure
$\mu$ on $\overline{M}$ such that $p(\phi)=\int\overline{\phi^M}d\mu$
for all $\phi(x)$.
This shows that $\zeta$ is surjective.
Let $p$ be extreme. Then, $\zeta^{-1}(p)$ is a face of $\mathcal P(\overline{M})$.
Let $\nu$ be an extreme point of $\zeta^{-1}(p)$.
Then, $\nu$ is an extreme point of $\mathcal P(\overline{M})$ too.
However, the extreme points of $\mathcal P(\overline{M})$ are pointed measures, i.e.
$\nu=\delta_a$ for some $a\in\overline{M}$ (see \cite{Aliprantis-Inf} Th 15.9).
We conclude that for every $\phi$
$$p(\phi)=\int\overline{\phi^M}\ d\delta_a=\overline{\phi^M}(a).$$
Now assume $a_k\in M$ and $a_k\rightarrow a$.
Let $p_k=tp(a_k)$. Then for each $\phi\in\mathbb D_1$
$$p_k(\phi)=\phi^M(a_k)\rightarrow\overline{\phi^M}(a)=p(\phi).$$
Therefore, $p_k\rightarrow p$.
\end{proof}
\bigskip

For $M\vDash T$, let $\mathbb{T}_M$ be the complete theory of $M$ and $\mathbb{S}_n(\mathbb{T}_M)$
be the type space of $\mathbb{T}_M$, both in the full continuous logic sense.
Note that if we allow $\wedge$ in the construction of formulas, a type $p\in\mathbb{S}_n(\mathbb{T}_M)$
is a normalized positive linear functional on the vector lattice of formulas which preserves $\wedge$.

\begin{corollary} \label{compact-realization1}
If $M\vDash T$ is compact, then every $p\in E_n(T)$ is realized in $M$.
More generally, if $M$ realizes every type in $\mathbb{S}_n(\mathbb T_M)$, then $M$ realizes every extreme type in $S_n(T)$.
\end{corollary}
\begin{proof}
For the first part, note that the range of $\a\mapsto tp(\a)$ is closed.
For the second part, let $\xi:\mathbb{S}_n(\mathbb T_M)\rightarrow S_n(T)$ be the restriction map,
i.e. $\xi(p)$ is the restriction of $p$ to the linear part of $p$. Clearly, $\xi$ is continuous.
By Proposition \ref{dense-realization} and the assumption, the closure of $\xi(\mathbb{S}_n(\mathbb T_M))$
contains $E_n(T)$. Since $\mathbb{S}_n(\mathbb T_M)$ is compact,
we conclude that $M$ realizes all extreme types $S_n(T)$.
\end{proof}
\vspace{1mm}

In particular, if $T$ has a first order model $M$ and $\mathcal F$ is a countably incomplete ultrafilter
on a set $I$, then $M^{\mathcal F}$ realizes all types in $\mathbb{S}_n(\mathbb T_M)$ for every $n$.
So, all extreme types of $T$ are realized in a first order model of $T$.
If, furthermore, the first order theory of $M$ is $\aleph_0$-categorical, then $M$ realizes
all extreme types in $S_n(T)$. Hence, $E_n(T)$ is finite in this case.
\vspace{1mm}

A model $M$ is \emph{extremally $\aleph_1$-saturated} if for every countable $A\subseteq M$,
all extreme types in $S_n(A)$ are realized in $M$.
Every compact model is extremally $\kappa$-saturated for every $\kappa$.
Below, $\mathcal{U}(M)$ denotes the set of all ultracharges on $M$.

\begin{proposition}\label{sat1}
Let $\mathcal F$ be a countably incomplete ultrafilter on $I$ and for each $i\in I$,
$M_i$ be an $L$-structure where $L$ is countable. Then, $M=\prod_{\mathcal F}M_i$ is extremally
$\aleph_1$-saturated.
\end{proposition}
\begin{proof}
If $a^1,a^2,...\in M$, then $(M,a^1,a^2,...)\simeq\prod_{\mathcal F}(M_i,a^1_i,a^2_i,...)$.
So, we may forget the parameters and show directly that every extreme type
$p(\x)$ of $Th(M)$ is realized in $M$. For simplicity assume $|\x|=1$.
Let $$V=\big\{\wp\in\mathcal{U}(M)|\ \ \ p(\phi)=\int\phi^M(x)d\wp\ \ \ \ \forall\phi\big\}.$$
We may consider $p$ as a function defined on the space of functions $\phi^M(x)$.
Then, by Kantorovich extension theorem (\cite{Aliprantis-Inf}, Th. 8.32)
$p$ extends to a positive linear functional $\bar p$ on $\ell^{\infty}(M)$.
Then, $\bar p$ is represented by integration over an ultracharge on $M$
so that $V$ is non-empty.
Moreover, $V$ is a closed face of $\mathcal{U}(M)$.
In particular, suppose $\lambda\mu+a(1-\lambda)\nu=\wp\in V$ where $\lambda\in(0,1)$.
Define types $p_\mu$, $p_\nu$ by setting for each $\phi(x)$,\
$p_\mu(\phi)=\int\phi^M d\mu$ and $p_\nu(\phi)=\int\phi^M d\nu$.
Then, $\lambda p_\mu+(1-\lambda)p_\nu=p$.
We have therefore that $p_\mu=p_\nu=p$ and hence $\mu,\nu\in V$.

Let $\wp$ be an extreme point of $V$.
Then, $\wp$ is an extreme point of $\mathcal{U}(M)$ and hence corresponds to
an ultrafilter, say $\mathcal D$ (not to be confused with the ultrafilter $\mathcal F$ on $I$).
We have therefore that $$\ \ \ \ p(\phi)=\int_M\phi^M(x)d\wp=
\lim_{\mathcal D,x}\phi^M(x) \ \ \ \ \ \ \ \forall\phi.\hspace{12mm} (*)$$
Since the language is countable, $p$ is axiomatized by a countable set of conditions
say $$p(x)\equiv\{0\leqslant\phi_1(x),\ 0\leqslant\phi_2(x),\ \ldots\ \}.$$
Let $$X_n=\big\{i\in I|\ \ \ -\frac{1}{n}<\sup_x\bigwedge_{k=1}^n\phi_k^{M_i}(x)\ \big\}.$$
Since $\mathcal D$ is an ultrafilter, there exists $a\in M$ such that
$-\frac{1}{n}<\phi^M_k(a)$ for $k=1,...,n$.
We have therefore that $X_n\in\mathcal F$.
Let $I_1\supseteq I_2\supseteq\cdots$ be a chain such that
$I_n\in\mathcal F$ and $\bigcap_n I_n=\emptyset$.
Then, $Y_n=I_n\cap X_n\in\mathcal F$.
Also, $Y_n$ is decreasing and $\bigcap_n Y_n=\emptyset$.
For $i\not\in Y_1$ let $a_i\in M_i$ be arbitrary.
For $i\in Y_1$, take the greatest $n_i$ such that $i\in Y_{n_i}$
and let $a_i\in M_i$ be such that $-\frac{1}{n_i}\leqslant\phi_k^{M_i}(a_i)$
for $k=1,...,n_i$. Let $a=[a_i]$. Then, for every $n$, if $i\in Y_n$,
we have that $n\leqslant n_i$ and hence $-\frac{1}{n}\leqslant\phi_k^{M_i}(a_i)$ for $k=1,...,n$.
We conclude that $0\leqslant\phi_k^M(a)$ for every $k\geqslant1$, i.e. $a$ realizes $p$.
\end{proof}
\bigskip

Now, we extend a bit the framework of logic and allow arbitrary (maybe uncountable) sets of individual variables.
Every formula uses a finite number of variables as before.
If $\x$ is a (possibly infinite) tuple of variables, then $D_{\x}(T)$ is the normed space of all
formulas (up to $T$-equivalence) whose free variables are included in $\x$.
Also, $S_{\x}(T)$ is the compact convex set consisting of all complete types on $\x$, i.e.
positive linear functionals $p:D_{\x}(T)\rightarrow\Rn$ with $p(1)=1$.

Let $\Gamma(\x)$ be a set of conditions satisfiable with $T$ (a partial type).
Complete types can be still regarded as partial types.
We say $\Gamma$ is a face of $S_{\x}(T)$ if the set $$\{p\in S_{\x}(T):\ \Gamma\subseteq p\}$$ is a face of $S_{\x}(T)$.
Note that if $\Gamma(\x,\y)$ is satisfiable and $\a$ realizes $\Gamma|_{\x}$ then $\Gamma(\a,\y)$ is satisfiable.

\begin{proposition}\label{ext2ext1}
Let $\Gamma(\x,\y)$ be a face of $S_{\x\y}(T)$. Then $\Gamma|_{\x}$ is a face of $S_{\x}(T)$.
In particular, the image of an extreme type under the restriction map $S_{n+1}(T)\rightarrow S_n(T)$ is extreme.
\end{proposition}
\begin{proof} Let $M\vDash T$ be $\kappa$-saturated where $|\x|+|\y|+\aleph_0<\kappa$.
Assume that $$\frac{1}{2}tp(\a_1)+\frac{1}{2}tp(\a_2)=tp(\a)\in\Gamma|_{\x}.\ \ \ \ \ \ \ (*)$$
where $\a_1,\a_2,\a\in M$.
Since $\Gamma(\a,\y)$ is satisfiable, by saturation, there exists $\b$ such that
$(\a,\b)\vDash\Gamma(\x,\y)$. Let $\Sigma(\bar{u},\bar{v})$ be the set of all conditions
of the form $$\frac{1}{2}\phi(\a_1,\bar{u})+\frac{1}{2}\phi(\a_2,\bar{v})=\phi^M(\a,\b).$$
where $\phi(\x,\y)$ is an $L$-formula.
$\Sigma$ is closed under linear combinations. Let
$$\theta(\bar{u},\bar{v})=\frac{1}{2}\phi(\a_1,\bar{u})+\frac{1}{2}\phi(\a_2,\bar{v}).$$
By the assumption $(*)$,
$$\inf_{\bar{u}\bar{v}}\theta^M(\bar{u},\bar{v})=\frac{1}{2}\inf_{\bar{u}}\phi^M(\a_1,\bar{u})+\frac{1}{2}\inf_{\bar{v}}\phi^M(\a_2,\bar{v})
=\inf_{\y}\phi^M(\a,\y)\leqslant\phi^M(\a,\b)$$
Similarly, one has that $\phi^M(\a,\b)\leqslant\sup_{\bar{u}\bar{v}}\theta^M(\bar{u},\bar{v})$.
This shows that $\theta(\bar{u},\bar{v})=\phi^M(\a,\b)$ is satisfiable in $M$.
By linear compactness, $\Sigma$ is satisfiable in $M$. Let $(\c,\e)\vDash\Sigma$.
Then, we have that $$\frac{1}{2}tp(\a_1,\c)+\frac{1}{2}tp(\a_2,\e)=tp(\a,\b)\in\Gamma(\x,\y).$$
Since $\Gamma$ is a face, we must have that $tp(\a_1,\c), tp(\a_2,\e)\in\Gamma(\x,\y)$.
Restricting to $\x$, we conclude that $tp(\a_1), tp(\a_2)\in\Gamma|_{\x}$.
\end{proof}
\bigskip

It is easy to verify that $\Gamma(\x)$ is a face if and only if $\Gamma|_{\y}$ is a face for every finite $\y\subseteq\x$.
In particular, $p(\x)$ is extreme if and only if $p|_{\y}$ is extreme for every finite $\y\subseteq\x$.

In a compact convex set, a point $p$ is exposed if it is the unique maximizer of a non-zero continuous
linear functional (see \cite{Aliprantis-Inf}, 7.15). Every exposed point is extreme.
Exposed types may be regarded as the linear variant of principal types.
However, extreme types have more flexibility and behave similarly.
We say a tuple $\a\in M$ is extreme over $A\subseteq M$ if $tp(\a/A)$ is extreme in $S_n(A)$.
In particular, $\a$ is extreme if it is extreme over $\emptyset$.
As stated above, if $\b$ realizes $p(\x,\y)|_{\y}$, then $p(\x,\b)$ is a type.
Also, if $(\a,\b)$ realizes $p(\x,\y)$, then
$$p(\x,\b)(\phi(\x,\b))=tp(\a/\b)(\phi(\x,\b))=\phi^M(\a,\b)=p(\phi(\x,\y)).$$

\begin{proposition} \label{extreme3}
$\a\b\in M$ is extreme if and only if $\b$ is extreme and $\a$ is extreme over $\b$.
In particular, if $\a$ is extreme, the restriction map $S_n(\a)\rightarrow S_n(\emptyset)$
takes extreme types to extreme types.
\end{proposition}
\begin{proof}
We may assume $M$ is $\aleph_0$-saturated and that $|\a|=|\b|=1$. Let $p(x,y)=tp(a,b)$.

$\Rightarrow$:\ \ That $b$ is extreme is a consequence of Proposition \ref{ext2ext1}.
Suppose that $$tp(a/b)=\frac{1}{2}p_1(x,b)+\frac{1}{2}p_2(x,b).$$
Assume $c,e\in M$ realize $p_1(x,b)$, $p_2(x,b)$ respectively.
Then, for each $\phi(x,y)$
$$p(\phi)=\phi^M(a,b)=tp(a/b)(\phi(x,b))=\frac{1}{2}\phi^M(c,b)+\frac{1}{2}\phi^M(e,b).$$
This means that 
$$p=\frac{1}{2}tp(c,b)+\frac{1}{2}tp(e,b).$$
Since, $p$ is extreme, $p(x,y)=tp(c,b)=tp(e,b)$ and hence $p(x,b)=p_1(x,b)=p_2(x,b)$.
\bigskip

$\Leftarrow$: Assume $p(x,y)=\frac{1}{2}p_1(x,y)+\frac{1}{2}p_2(x,y)$.
Since $b$ is extreme, by restricting these types to $y$,
we conclude that $tp(b)=p|_y=p_1|_y=p_2|_y$. Therefore, $p_1(x,b)$ and $p_2(x,b)$ are realizable types
and we have that
$$tp(a/b)=p(x,b)=\frac{1}{2}p_1(x,b)+\frac{1}{2}p_2(x,b).$$
Since $a$ is extreme over $b$, we have that
$p(x,b)=p_1(x,b)=p_2(x,b)$. Since this type is realized by $a$,
we conclude that $p=p_1=p_2$ which means that $p$ is extreme.
\end{proof}

\begin{lemma} \label{facetoface}
Let $\Gamma(\x,\y)$ be a face and $\b\in M$. If $\Gamma(\x,\b)$ is satisfiable, it is a face of $S_{\x}(\b)$.
\end{lemma}
\begin{proof} We may assume $M$ is $\aleph_0$-saturated and $|\x|=1$. Let $a$ realize $\Gamma(x,\b)$.
Suppose that $a_1,a_2\in M$ and $$\frac{1}{2}tp(a_1/\b)+\frac{1}{2}tp(a_2/\b)=tp(a/\b)\in\Gamma(x,\b).$$
This means that $$\frac{1}{2}tp(a_1,\b)+\frac{1}{2}tp(a_2,\b)=tp(a,\b)\in\Gamma(x,\y).$$
Since $\Gamma(x,\y)$ is a face, we conclude that $tp(a_1,\b),\ tp(a_2,\b)\in\Gamma(x,\y)$.
Hence, $$tp(a_1/\b),\ tp(a_2/\b)\in\Gamma(x,\b).$$
\end{proof}

\begin{proposition}
$\b$ is extreme over $\a$ if and only if $\a$ is extreme over $\b$.
\end{proposition}
\begin{proof}
Assume $\a$ is not extreme over $\b$. Let
$$tp(\a/\b)=\frac{1}{2}p_1(\x,\b)+\frac{1}{2}p_2(\x,\b).$$
Let $\Sigma(\bar{u},\bar{v})$ be the family of all conditions of the form
$$\frac{1}{2}\phi(\a,\bar{u})+\frac{1}{2}\phi(\a,\bar{v})=\phi^M(\a,\b).$$
$\Sigma$ is linearly closed and we have that
$$\inf_{\bar{u}\bar{v}}[\frac{1}{2}\phi^M(\a,\bar{u})+\frac{1}{2}\phi^M(\a,\bar{v})]=
\inf_{\bar{u}}\phi^M(\a,\bar{u})\leqslant\phi^M(\a,\b).$$
Similarly,
$$\phi^M(\a,\b)\leqslant\sup_{\bar{u}\bar{v}}[\frac{1}{2}\phi^M(\a,\bar{u})+\frac{1}{2}\phi^M(\a,\bar{v})].$$
So, $\Sigma$ is satisfiable, say by $(\c,\e)$. Then, we have that
$$tp(\b/\a)=\frac{1}{2}tp(\c/\a)+\frac{1}{2}tp(\e/\a).$$
This shows that $tp(\b/\a)$ is not extreme,
\end{proof}

\section{Omitting non-extreme types}
In the previous section we proved that all extreme types are realized in compact models (if any).
In this section we show that non-extreme types are omitted in suitable models of the theory.
The following remark is a consequence of general facts for compact convex sets (see \cite{Simon}).

\begin{remark}
If $\Gamma_i(\x)$ is a face of $S_{\x}(T)$ for each $i\in I$, then so is $\bigcup_{i\in I}\Gamma_i(\x)$.
If $\Gamma(\x)$ is a face and $\Gamma(\x)\vDash\theta(\x)\leqslant0$,
then $\Gamma(\x)\cup\{0\leqslant\theta(\x)\}$ is either a face or is empty.
If $\Gamma(\x)$ is a face of $S_{\x}(T)$ and $\y$ is disjoint from $\x$, then $\Gamma(\x,\y)=\Gamma(\x)$
is a face of $S_{\x\y}(T)$.
\end{remark}

\begin{theorem} \label{omitting}
Let $T$ be a complete theory in a language $L$ and $p(\z)\in S_{n}(T)$
be non-extreme. Then $p$ is omitted in a (complete) model $M$ of $T$.
\end{theorem}
\begin{proof}
For simplicity we assume $n=1$.
Let $\kappa\geqslant|L|$ be such that $\kappa^{\aleph_0}=\kappa$ and $N$ be a $\kappa$-saturated model of $T$.
Let $\{X,Y,Z,W\}$ be a partition of $\kappa$ into sets of cardinality $\kappa$.
Let $C=\{c_i|\ i\in Z\}$ be a set of distinct new constant symbols,
$\{\sigma_i|\ i\in X\}$ be an enumeration of $L(C)$-sentences
and $\{\phi_i(y)| \ i\in Y\}$ be an enumeration of $L(C)$-formulas with one free variable $y$.
Also, assume that all $\omega$-sequences of constant symbols of $C$ are enumerated by indices from $W$.
We construct a chain
$$\ \ \ \ T_{0}\subseteq T_1\subseteq\cdots \subseteq T_i\subseteq\cdots\ \ \ \ \ \ \ \ \ \ \ \ \ i<\kappa$$
of satisfiable extensions of $T$ such that for each $i$:

(I) $T_i=T\cup \Gamma_i(\bar{e})$ where $\bar{e}\in C$ (maybe of infinite length) and $|\Gamma_i|<\kappa$

(II) $\Gamma_i(\x)$ is a face in $S_{\x}(T)$ for $0<i$.
\bigskip

Set $T_{0}=T$ and for infinite limit $i$ set $\Gamma_i=\cup_{j<i} \Gamma_j$,\ \ $T_i=T\cup\Gamma_i$.
Note that (I), (II) hold.
Assume $T_i$ is defined and (I), (II) hold. We define $T_{i+1}$ according to the following cases:
\vspace{1mm}

- $i\in X$:\ There is a biggest $r$ such that $T_i,r\leqslant\sigma_i$
is satisfiable. So, $T_i\vDash\sigma_i\leqslant r$.
Let $\e$ be the tuple obtained by unifying the constants of $C$ used in $\Gamma_i$ and $\sigma_i$.
Let $$\Gamma_{i+1}(\e)=\Gamma_i(\e)\cup\{r\leqslant\sigma_i(\e)\},\ \ \ \ \ \ \ T_{i+1}=T\cup\Gamma_{i+1}(\e).$$
Note that $T\cup\Gamma_i(\x)\vDash\sigma_i(\x)\leqslant r$ so that $\Gamma_{i+1}(\x)$  is a face of $S_{\x}(T)$.
\bigskip

- $i\in Y$:\ Unifying the constants of $C$ used in $\Gamma_i$ and $\phi$, we may write $\Gamma_i=\Gamma_i(\e)$ and $\phi_i=\phi_i(\e,y)$.
Take a $c\in C$ not been used in $T_i,\phi_i$.
Let $$\Gamma_{i+1}(\e,c)=\Gamma_i(\e)\cup\{\phi_i(\e,c)\leqslant\inf_y\phi_i(y)\},\ \ \ \ \ \ \ T_{i+1}=T\cup\Gamma_{i+1}(\e,c).$$
Then, the conditions (I), (II) are satisfied for $T_{i+1}$ and $\Gamma_{i+1}$.
\bigskip

- $i\in Z$:\ Assume $T_i=T\cup\Gamma_i(\bar{e},c_i)$ and conditions (I), (II) are satisfied.
We claim that there are $\theta(z)\leqslant0\in p(z)$ and $\epsilon>0$
such that  $T_i\cup\{\epsilon\leqslant\theta(c_i)\}$ is satisfiable.
Suppose not. Then, we must have that $$T\cup\Gamma_i(\x,z)\vDash p(z)$$
By Proposition \ref{ext2ext1}, $p(z)$ is an extreme type. This is a contradiction.
We conclude that there are $\theta(z)\leqslant0\in p(z)$ and a greatest $\epsilon>0$ such that and
$T_i\cup\{\epsilon\leqslant\theta(c_i)\}$ is satisfiable.
Let $$\Gamma_{i+1}(\e,c_i)=\Gamma_i(\e,c_i)\cup\{\epsilon\leqslant\theta(c_i)\},\ \ \ \ \ \ \ T_{i+1}=T\cup\Gamma_{i+1}(\e,c_i).$$
Then, the conditions (I), (II) are satisfied for $T_{i+1}$ and $\Gamma_{i+1}$.
\bigskip

- $i\in W$:\ Suppose that $T_i=T\cup\Gamma_i(\e)$ is constructed and that the
index $i$ corresponds to the sequence $(c_{i_n})$.
Suppose that for all $n,k$ $$d(c_{i_n},c_{i_{n+k}})\leqslant2^{-n}\in T_i$$
(hence every $c_{i_k}$ is a component of $\e$).
Let $c\in C$ be a constant symbol different from the components of $\e$. Set
$$\Gamma_{i+1}(\e,c)=\Gamma_i(\e)\cup\big\{d(c_{i_n},c)\leqslant 2^{-n}:\ \ n=1,2,...\big\}.$$
Note that $\Gamma_{i+1}(\e,c)$ is finitely and hence totally satisfiable.
We show that $\Gamma(\x,y)$ is a face of $S_{\x y}(T)$.
Assume $$\frac{1}{2}q_1(\x,y)+\frac{1}{2}q_2(\x,y)=q(\x,y)\supseteq\Gamma_{i+1}(\x,y).$$
Since $\Gamma_i(\x)$ is a face, $q_1|_{\x}$ and $q_2|_{\x}$ must contain it.
Suppose that $c_{i_n}$ corresponds to the variable $x_n\in\x$.
So, the condition $d(x_{n},x_{n+k})\leqslant2^{-n}$ belongs to both $q_1|_{\x}$ and $q_2|_{\x}$ for each $n$.
Suppose that $q_1(\x,y)$ is realized by a tuple in $N$ where $x_n=a_n$, $y=a$ and $d(a_n,a)=r_n$.
Also, $q_2(\x,y)$ is realized by a tuple where $x_n=b_n$, $y=b$ and $d(b_n,b)=s_n$.
By the above assumption, we must have that $\frac{1}{2}r_n+\frac{1}{2}s_n\leqslant2^{-n}$ for each $n$.
This shows that $a_n\rightarrow a$ and $b_n\rightarrow b$.
Since $d(a_n,a_{n+k})\leqslant2^{-n}$ for each $k$, we must have that $d(a_n,a)\leqslant2^{-n}$.
Hence, $d(x_n,y)\leqslant2^{-n}$ belongs to $q_1(\x,y)$ for each $n$.
This implies that $q_1(\x,y)\supseteq\Gamma_{i+1}(\x,y)$. Similarly, $q_2(\x,y)\supseteq\Gamma_{i+1}(\x,y)$.
\vspace{2mm}

Finally, let $\overline T=\cup_i T_i$.
This is a complete $L(C)$-theory with the following properties:

- for every $\phi(y)$ in $L(C)$, there exists $c\in C$ such that $\phi(c)\leqslant\inf_y\phi(y)\in\overline T$

- for every $c\in C$, there exists $\theta(z)\leqslant0\in p(z)$ and $\epsilon>0$ such that
$\epsilon\leqslant\theta(c)\in\overline T$.

- for every sequence $(c_n)$ of constant symbols from $C$,
if $d(c_n,c_{n+k})\leqslant2^{-n}\in\overline T$ for every $k,n$, then there exists $c\in C$
such that $d(c_n,c)\leqslant2^{-n}\in\overline T$ for every $n$.
\vspace{2mm}

We define the canonical model of $\overline T$ as follows.
For $c,e\in C$ set $c\sim e$ if $d(c,e)=0\in\overline T$.
The equivalence class of $c$ under this relation is denoted by $\hat c$.
Let $M=\{\hat c:\ c\in C\}$. For $\hat c,\hat e\in M$ set $d^M(\hat c,\hat e)=r$
if $d(c,e)=r\in\overline T$. This defines a metric on $M$.
Define an $L(C)$-structure on $M$ by setting for all constant, relation and function symbols
(say unary for simplicity):

- $c^M=\hat c$\ \ if $c\in C$\ \ and \ \ $c^M=\hat e$\ \ if $c\in L$, $e\in C$ and $d(c,e)=0\in \overline T$

- $R^M(\hat c)=r$ \ if \ $R(c)=r\in\overline T$

- $F^M(\hat c)=\hat e$ \ if \ $d(F(c),e)=0\in\overline T$.
\vspace{1mm}

Note that for $c\in L$, since $\inf_x d(c,x)=0$ is satisfied in every model, there exists
$e\in C$ such that $d(c,e)=0\in\overline T$.
Similarly, for $c\in C$, there exists $e\in C$ such that $d(F(c),e)=0\in\overline T$.
It is not hard to verify that $M$ is a well-defined $L(C)$-structure.
It is routine to show by induction on the complexity of formulas that
for every $L$-formula $\phi(x_1,...,x_n)$, $r\in\Rn$ and $c_1,...,c_n\in C$
$$\phi^M(\hat{c}_1,...,\hat{c}_n)=r \ \ \ \ \ \ \
\mbox{iff}\ \ \ \ \ \ \ \phi(c_1,...,c_n)=r\in\overline T.$$
Another way to obtain the canonical model is to consider a model $M'\vDash\overline{T}$
and to show that $M=\{c^{M'}:\ c\in C\}$ is an elementary submodel of $M'$.

It is also clear by the construction of $\overline T$ that $M$
is metrically complete and that it omits $p(z)$.
\end{proof}
\bigskip

A stronger result is obtained by simultaneously omitting all non-extreme types.

\begin{theorem} \label{strong omitting}
There is a model $M\vDash T$ omitting every non-extreme type in every $S_n(T)$.
\end{theorem}
\begin{proof}
There is nothing to prove if $T$ trivial. Otherwise, $T$ has non-extreme types.
Let $\{p_\alpha(z_1,...,z_{n_\alpha}):\ \alpha\in\lambda\}$ be an enumeration of all
non-extreme types. Let $\kappa\geqslant|L|+\lambda$ be such that $\kappa^{\aleph_0}=\kappa$
and $C$ be a set of new constant symbols of cardinality $\kappa$.
Let $$\{X,Y,W\}\cup\{Z_\alpha:\ \alpha<\kappa\}$$ be a partition of $\kappa$
into $\kappa$ disjoint sets of cardinality $\kappa$.
Let $\{\sigma_i|\ i\in X\}$ be an enumeration of $L(C)$-sentences,
$\{\phi_i(y)| \ i\in Y\}$ be an enumeration of $L(C)$-formulas with one free variable $y$
and for each $\alpha$, $\{\c_i:\ i\in Z_\alpha\}$ be an enumeration of all
$n_\alpha$-tuples of constant symbols from $C$.
Also, assume that all $\omega$-sequences of constant symbols from $C$ are enumerated by indices from $W$.
Then the argument follows as in Theorems \ref{omitting} and the resulting theory
has a metrically complete canonical model omitting every $p_\alpha$.
\end{proof}
\vspace{1mm}

The model given by Theorem \ref{strong omitting} omits non-extreme types with infinite number of variables too.
This is because $p(\x)$ is extreme if and only if $p|_{\z}$ is extreme for every finite $\z\subseteq\x$.

A model $M\vDash T$ is \emph{extremal} if every $\a\in M$ has an extreme type.
So, omitting types theorem states that every complete theory has an extremal model.
If $M$ is extremal, then for every $A\subseteq M$, the structure $(M,a)_{a\in A}$ is extremal.

Let us call a theory $T$ \emph{ample} if it has a model $M$ with an infinite sequence $a_n\in M$
such that $d(a_m,a_n)=1$ for $m\neq n$ (note that we assume $d(x,y)\leqslant1$ for all $x,y$).
If $T$ is ample, it has arbitrarily large models omitting the non-extreme type $p(\z)$.

\begin{proposition} \label{extremal3}
Let $T$ be a complete ample theory in a language $L$.
Let $\lambda$ be an infinite cardinal and $p(\z)\in S_n(T)$ be non-extreme.
Then $p$ is omitted in a model $M$ with $\lambda\leqslant |M|$.
\end{proposition}
\begin{proof}
Let $\kappa>\lambda+|L|$ be such that $\kappa^{\aleph_0}=\kappa$.
Take a partition of $\kappa$ as in the proof of Theorem \ref{omitting}.          
Let $D$ be a subset of $C$ of cardinality $\lambda$.                             
Let $$\Gamma_{0}=\{1\leqslant d(c,c'):\ c,c'\in D\ \mbox{are\ distinct}\},
\ \ \ \ \ \ \ \ T_{0}=T\cup\Gamma_{0}.$$
Then, $T_0$ is consistent and $T_0$, $\Gamma_0$ satisfy the conditions (I), (II)
in the proof of Theorem \ref{omitting}.
We then continue the construction of $T_i$ and $\Gamma_i$ as before.
The resulting canonical model has cardinality at least $\lambda$ and omits $p(\z)$.
\end{proof}

Similarly, one can prove that

\begin{proposition} \label{extremal4}
Let $T$ be a complete ample theory in a language $L$ and $\lambda$ be an infinite cardinal.
Then there is a model $M$ which omits all non-extreme types and $\lambda\leqslant |M|$.
\end{proposition}
\vspace{1mm}

\begin{proposition} \label{extremal5}
Let $p(\z)$ be extreme. Then there is a model $K\vDash T$ which realizes $p$ and omits every non-extreme type of $T$.
\end{proposition}
\begin{proof}
Let $\a\in M\vDash T$ realize $p(\z)$ and $\bar T=Th(M,\a)$.
Let $N$ be an extremal model of $\bar T$.
and $K$ be its reduction to the language of $T$.
Then, $K$ realizes $p(\z)$. On the other hand, for each $\b\in K$,\ \ $tp(\b/\a)$ is extreme.
So, since $\a$ is extreme, $\a\b$ is extreme. This implies that $\b$ extreme.
So, $K$ omits every non-extreme type of $T$.
\end{proof}

The following example shows that there are extreme types which can be omitted.

\begin{example}
\emph{Let $L=\{c_0,c_1,...\}$ and $T$ be the theory of $M=\{a_0,a_1,...\}$ where
$c_i^M=a_i$ and $d(a_i,a_j)=1$ for $i<j$.
Let $\Sigma(x)=\{d(x,c_i)=1:\ i\in\omega\}$. Then, $\Sigma$ is a face.
In particular, assume $\frac{1}{2}p_1+\frac{1}{2}p_2\supseteq\Sigma$ where $p_1$ is realized by $b_1\in N$
and $p_2$ is realized by $b_2\in N$ and $M\preccurlyeq N$.
Then, $\frac{1}{2}d(b_1,c_i)+\frac{1}{2}d(b_2,c_i)=1$ for each $i$. This implies that $d(b_1,c_i)=d(b_2,c_i)=1$.
Let $p(x)\supseteq\Sigma$ be extreme. Then, $p$ is not realized in $M$.}
\end{example}

\section{Extremal models}
In this section we give some applications of the omitting types theorem,
mostly in the cases where the theory has a compact or first order model.
By CL-theory, CL-equivalence etc. we mean the corresponding notion in the sense of full continuous logic.

A partial isomorphism between $M$ and $N$ is a relation $\a\sim\b$ where $\a\in M$, $\b\in N$
have the same length such that (i) if $\a\sim\b$ then $\theta^M(\a)=\theta^N(\b)$ for every atomic $\theta$
(ii) if $\a\sim\b$, then for every $c\in M$ there exists $e\in N$ such that $\a c\sim\b e$
(ii) if $\a\sim\b$, then for every $e\in N$ there exists $c\in M$ such that $\a c\sim\b e$.

It is easily verified that if $M$ and $N$ are partially isomorphic then $M\equiv_{\mathrm{CL}} N$.

\begin{proposition} \label{extremal embedding}
Assume $L$ is countable. Then, $E_n(T)$ is closed for all $n$ if and only if
the class of extremal models of $T$ is axiomatized by a CL-theory $\mathbb T$.
In this case, $\mathbb T$ is CL-complete.
\end{proposition}
\begin{proof}
Assume every $E_n(T)$ is closed. We show that the class of extremal models is closed under ultraproduct and CL-equivalence.
Let $\mathcal F$ be an ultrafilter on a set $I$ and $M_i\vDash T$ be extremal for each $i$.
Let $a=[a_i]\in M=\prod_{\mathcal F}M_i$ and $p_i(x)=tp(a_i)\in E_1(T)$.
Then $$p=\lim_{\mathcal{F},i}p_i\in E_1(T)$$ and for each $\phi(x)$
$$p(\phi)=\lim_{\mathcal{F},i} p_i(\phi)=\lim_{\mathcal{F},i}\phi^{M_i}(a_i)=\phi^{M}(a).$$
So, every $a\in M$ (and similarly every tuple $\a\in M$) has an extreme type.
On the other hand, assume $M\vDash T$ is extremal and $N\equiv_{\mathrm{CL}}M$.
Then, there is an ultrafilter $\mathcal F$ such that $N$ is elementarily embedded in $M^{\mathcal F}$. So, $N$ is extremal.
We conclude that there exists a CL-theory axiomatizing the extremal models of $T$.
Conversely assume the extremal models of $T$ are axiomatized by a CL-theory $\mathbb T$.
Let $\xi: \mathbb{S}_n(\mathbb T)\rightarrow S_n(T)$ be the restriction map.
Then, $\xi$ is continuous and its range is a subset of $E_n(T)$.
Let $M\vDash\mathbb{T}$ and $\mathcal F$ be countably incomplete. Then, $M^{\mathcal F}\vDash T$
is extremal as well as extremally $\aleph_1$-saturated.
So, the range of $\xi$ is exactly the set of extreme types. We conclude that every $E_n(T)$ is closed.

For the second part of the proposition, assume $M,N\vDash\mathbb{T}$. We show that $M\equiv_{\mathrm{CL}}N$.
By the CL variant of downward L\"{o}wenheim-Skolem theorem,
we may assume $M$ and $N$ are both separable. Let $\mathcal F$ be a countably incomplete ultrafilter on $\Nn$.
Then, $M^{\mathcal F}$ and $N^{\mathcal F}$ are extremal and extremally $\aleph_1$-saturated models of $T$.
We show that they are partially isomorphic.
For $\a\in M^{\mathcal F}$ and $\b\in N^{\mathcal F}$ of the same length set $\a\sim\b$
if $tp(\a)=tp(\b)$.
Suppose $\a\sim\b$ and $c\in M^{\mathcal F}$ is given. Then, $tp(c/\a)$ is extreme.
So, its shift $$\{\phi(\b,x)=r:\ \ \phi^{M^{\mathcal F}}(\a,c)=r\}$$ is extreme.
Assume it is realized by $e\in N^{\mathcal F}$. Then, $\a c\sim\b e$. The back property is verified similarly.
We conclude that $M\equiv_{\mathrm{CL}}N$.
\end{proof}
\vspace{1mm}

Countability of the language in the above proposition (as well as the following ones)
is because we use Proposition \ref{sat1}. Assuming a result similar to Th. 6.1.8 of
\cite{Chang-Keisler1} is proved for extremal types, this condition can be removed.
If every $E_n(T)$ is closed, the CL-theory of extremal models of $T$ is denoted by $T_e$.
Our next goal is to show that $E_n(T)$ coincides with $\mathbb{S}_n(T_e)$.

For $p,q\in E_n(T)$ let $r=d(p,q)$ and $$\Sigma(\x,\y)=p(\x)\cup q(\y)\cup\{d(\x,\y)\leqslant r\}.$$
$\Sigma$ is satisfiable. We show that it is a face. Suppose that $\frac{1}{2}p_1(\x,\y)+\frac{1}{2}p_2(\x,\y)\supseteq\Sigma$.
Restricting $p_1$ and $p_2$ to $\x$ (and then to $\y$), we conclude that $p_1$ and $p_2$ contain $p(\x)\cup q(\y)$.
Suppose that $(\a_1,\a_2)$ realizes $p_1$ and $(\b_1,\b_2)$ realizes $p_2$. Then, we must have that
$$\frac{1}{2}d(\a_1,\a_2)+\frac{1}{2}d(\b_1,\b_2)\leqslant r.$$
This implies that $d(\a_1,\a_2)=d(\b_1,\b_2)=r$ and hence $p_1$, $p_2$ contain $\Sigma$.
Since every face contains an extreme point, we conclude that if $K$ realizes all extreme types, then
$$d(p,q)=\inf\{d(\a,\b):\ \a\vDash p,\ \b\vDash q,\ \a,\b\in K\}.$$
In particular, if $T$ has a first order model $M$, then all extreme types are realized in
$K=M^{\mathcal F}$ for some suitable $\mathcal F$.
So, in this case, $d(p,q)=1$ for every distinct extreme types $p,q$.

\begin{proposition} \label{complete}
Assume $L$ is countable and every $E_n(T)$ is closed.
Then, the restriction map $\xi:\mathbb{S}_n(T_e)\rightarrow E_n(T)$ is an
isometry of the metric topologies and homeomorphism of the logic topologies.
\end{proposition}
\begin{proof}
Since every $p\in E_n(T)$ is realized in an extremal $M\vDash T$ which is also a model of $T_e$, $\xi$ is surjective.
Let $\a,\b\in M$ where $M$ is extremal and $tp(\a)=tp(\b)$.
Let $\mathcal F$ be an countably incomplete ultrafilter on $\Nn$.
So, $M\preccurlyeq_{\mathrm{CL}} M^{\mathcal F}\vDash T_e$.
Then, as in the proof of Proposition \ref{extremal embedding}, $(M^{\mathcal F},\a)$ and $(M^{\mathcal F},\b)$
are partially isomorphic. So, $\a$ and $\b$ have the same CL-type.
This proves that $\xi$ is injective.
Now, the above explanation shows that $\xi$ is an isometry. Also, since $\xi$ is logic-continuous
and $\mathbb{S}_n(T_e)$ is compact, $\xi$ is a homeomorphism of the logic topologies.
\end{proof}

We recall that by \cite{BBHU} Theorem 12.10 a complete CL-theory $\mathbb T$ is $\aleph_0$-categorical
if and only if every $\mathbb{S}_n(\mathbb{T})$ is compact in the metric topology. The theorem includes
the case where $\mathbb T$ has a (unique) compact model if we call such a theory $\aleph_0$-categorical.

\begin{proposition}
Let $L$ be countable. Then, every $E_n(T)$ is compact in the metric topology if and only if
every $E_n(T)$ is closed in the logic topology and $T_e$ is $\aleph_0$-categorical.
\end{proposition}
\begin{proof}
Assume $E_n(T)$ is compact in the metric topology. Then, it is compact in the logic topology.
So, it is a closed subset of $S_n(T)$.
By Proposition \ref{complete}, $\mathbb{S}_n(T_e)$ is isometric to $E_n(T)$.
Therefore, $T_e$ is $\aleph_0$-categorical.
Conversely, if every $E_n(T)$ is closed and $T_e$ is $\aleph_0$-categorical,
then $E_n(T)$ is isometric to $\mathbb{S}_n(T_e)$ which is compact in the metric topology.
\end{proof}

\begin{proposition}
Let $L$ be countable and $T$ be a linear theory which is CL-complete, i.e. $M\equiv_{\mathrm{CL}}N$
for any $M,N\vDash T$. If every $E_n(T)$ is closed, then $T$ is trivial.
\end{proposition}
\begin{proof}
$T_e$ is the CL-theory of any extremal model $K$ of $T$. By the assumption,
for every $M\vDash T$ one has that $M\equiv_{\mathrm{CL}}K$. So, every model of $T$ is extremal.
This shows that $T$ has no non-extreme type. This happens only if the singleton is the only model of $T$.
\end{proof}

\begin{example}
\emph{Let APrA be the theory probability algebras with an aperiodic automorphism \cite{BBHU}.
This is a CL-complete theory.
It is shown in \cite{Bagheri-Lip} that APrA can be axiomatized by linear axioms.
So, regarding it as a non-trivial linear theory which is CL-complete, we conclude that
$E_n(\mathrm{APrA})$ is not closed for some $n$.}
\end{example}
\vspace{1mm}

Although extremal models do not form an elementary class in general, it is easy to verify
that they form an abstract elementary class (see \cite{Vilaveces} for definition).
Below we prove that this class has the elementary joint elementary embedding property.
Elementary amalgamation property is proved similarly.

As for types, complete theories may be regarded as normal positive linear functionals on the normed space of sentences.
A complete theory $\overline{T}$ is an \emph{extremal extension} of an incomplete theory $T_0$ if it is an extremal
point of the compact convex set $$\{T\supseteq T_0:\ T\ \mbox{is\ complete} \}.$$
The notion of a type for $T_0$ is defined in the usual way. Then, complete types for $T_0$
(with free variables $\x$) form a compact convex set which is dented by $S_{\x}(T_0)$.

\begin{proposition}
Let $A,B\vDash T$ be extremal. Then, there exists an extremal $K$ such that $A\preccurlyeq K$ and $B\preccurlyeq K$.
\end{proposition}
\begin{proof}
Let $$T_0={\mbox ediag}(A)\cup\mbox{ediag}(B).$$
By linear compactness, $T_0$ has a model.
Let $\overline T$ be a complete extremal extension of $T_0$ (in the language $L(A\cup B)$).
Let $\overline K$ be an extremal model of $\overline T$ and $K$ be its restriction to the language of $T$.
We have that $A\preccurlyeq K$ and $B\preccurlyeq K$. We show that $K$ is an extremal model of $T$.
Let $\e\in K$. Then, $\e$ is extreme over $A\cup B$. We show that $A\cup B$ is extreme.
Then, by \ref{extreme3}, we conclude that $\e$ is extreme over $\emptyset$.
So, assume $A$ is enumerated by a tuple $\a$ and $B$ is enumerated by a tuple $\b$.
To prove that $\a\b$ is extreme, we have only to show that $\a$ is extreme over $\b$.
Assume $$tp(\a/\b)=p(\x,\b)=\frac{1}{2}p_1(\x,\b)+\frac{1}{2}p_2(\x,\b). \ \ \ \ \ \ \ (*)$$
Let $q, q_1,q_2$ be the restrictions of $p,p_1,p_2$ to $L$ respectively.
Since $\a$ is extreme, we have that $q=q_1=q_2$.
Let us identify every $a_i\in \a$ with the corresponding constant symbol $c_{a_i}$
(similarly for $\b$). Then, $p(\a,\b)$ is in fact the $L(A\cup B)$-theory $\overline{T}$.
Similarly, $p_1(\a,\b)$ and $p_2(\a,\b)$ are other complete $L(A\cup B)$-theories containing $T_0$
which we may denote by $T_1$ and $T_2$ respectively.
Then, by $(*)$ above, we have that $T=\frac{1}{2}T_1+\frac{1}{2}T_2$.
Since $T$ is extreme, we must have that $T=T_1=T_2$.
We then conclude that $p(\x,\b)=p_1(\x,\b)=p_2(\x,\b)$ and hence $\a$ is extreme over $\b$.
\end{proof}
\vspace{1mm}

We call a model $K$ \emph{minimal} if it has no proper elementary submodel.

\begin{lemma}
Every compact model $M$ has a minimal elementary submodel.
\end{lemma}
\begin{proof}
Let $\{K_i\}_{i\in I}$ be a maximal decreasing elementary chain of submodels of $M$.
In fact, $I$ is countable and there is no harm if we assume $I=\Nn$ (take a cofinal subset).
Let $K=\cap_nK_n$.
We show by induction on the complexity of formulas that $\phi^K=\phi^M$ for every $\phi$ with parameters in $K$.
The atomic and connective cases are obvious. Consider the case $\sup_x\phi(x)$.
Since the chain is elementary, there are $r,s$ such that for all $n$
$$r=\sup_x\phi^K(x)\leqslant\sup_x\phi^{K_n}(x)=s.$$
If $r<s$, there exists $a_n\in K_n$ such that $\frac{r+s}{2}\leqslant\phi^{K_n}(a_n)$.
Then $(a_n)$ has a subsequence converging to say $a$.
Moreover, $a\in K$ and $\frac{r+s}{2}\leqslant\phi^M(a)$. This is a contradiction.
We have therefore that $K\preccurlyeq M$ and $K$ is minimal.
\end{proof}
\vspace{1mm}

\begin{example} \label{sphere}\emph{
(i) We show that the unit circle $\mathbf{S}\subseteq\Rn^2$ equipped with the Euclidean metric is minimal
(temporarily, we allow metric spaces of diameter 2).
Let $K\preccurlyeq\mathbf{S}$ be minimal and $a\in K$. Then $d(x,a)$ is maximized by $-a$.
So, $-a\in K$. Also, there is $b\in K$ which maximizes $d(x,a)+d(x,-a)$ (with value $2\sqrt{2}$). So, $-b\in K$.
Generally, for each distinct $a,b\in K$, the midpoints of the arcs joining $a$ to $b$
are obtained as the maximizer of $d(x,a)+d(x,b)$ and its antipode.
We conclude that $K$ contains a dense subset of $\mathbf{S}$.
Since $K$ is closed, we must have that $K=\mathbf{S}$.\\
(ii) Similar argument shows that the unit disc $\mathbf B\subseteq\Rn^2$ equipped with the Euclidean metric is minimal.}
\end{example}
\bigskip

If $f:M\rightarrow N$ is a partial elementary map with domain $A$ and $p(\x)\in S_n(A)$,
then its shift $$f(p)(\phi(\x,f(\a)))=p(\phi(\x,\a))$$
is a type over $f(A)$. If $p$ is extreme, then $f(p)$ is extreme.

\begin{corollary} \label{minimal model}
If $T$ has a compact model, then $M\vDash T$ is separable extremal if and only if it is compact minimal.
Moreover, $T$ has a unique such model.
\end{corollary}
\begin{proof}
By Proposition \ref{strong omitting} and downward L\"{o}wenheim-Skolem theorem,
$T$ has separable extremal model $M$.
Let $K\vDash T$ be compact minimal. We show that $M\simeq K$.
Let $\{a_0,a_1,...\}\subseteq M$ be dense.
Let $f_{-1}=\emptyset$ and assume a partial elementary map
$f_i:\{a_0,...,a_i\}\rightarrow K$ is defined. Since $(a_0,...,a_{i+1})$
is extreme, $p(x)=tp(a_{i+1}/a_0...a_i)$ is extreme.
So, its shift $f_i(p)$ is an extreme type over $f_i(a_0)... f_i(a_i)$ realized by say $b\in K$.
Therefore, $f_{i+1}=f_i\cup\{(a_{i+1},b)\}$ is elementary.
Let $f$ be the continuous extension of $\cup_i f_i$ to all $M$.
Then, $f:M\rightarrow K$ is elementary. Since $K$ is minimal, we have that $M\simeq K$.
\end{proof}

It is well-known that if a compact convex set is metrizable, its set of extreme points is a $G_\delta$ set.
If $T$ has a compact model, its minimal model $K$ realizes exactly the extreme types.
So, $E_n(T)$ is a closed subset of $S_n(T)$.
Similarly, since $K^\omega$ is compact, $E_\omega(T)\subseteq S_\omega(T)$ is closed.
Also, since that map $\a\mapsto tp(\a)$ is $1$-Lipschitz, $E_n(T)$ is compact with the metric topology.
So, by maximality of compact Hausdorff topologies, metric and logic topologies must coincide on every $E_n(T)$.
\bigskip

\begin{proposition}\label{ext-first}
Assume $L$ is countable. If $T$ has a first order model, every extremal model of $T$ is first order.
If $T$ has a compact model, it has a unique extremal model namely its minimal one.
\end{proposition}
\begin{proof}
Let $M\vDash T$ be first order and $\mathcal F$ be a countably incomplete ultrafilter on $\Nn$.
So, $M^{\mathcal F}$ is extremally $\aleph_1$-saturated. Let $K\vDash T$ be extremal.
We have only to show that every separable elementary submodel of $K$ is embedded in $M^{\mathcal F}$.
So, assume $K$ itself is separable. Let $\{a_0,a_1,...\}\subseteq K$ be dense. Assume a partial elementary map
$f_i:\{a_0,...,a_i\}\rightarrow M^{\mathcal F}$ is defined.
Since $tp(a_{i+1}/a_0\ldots a_i)$ is extremal, there exists $b\in M^{\mathcal F}$
such that $f_i\cup\{(a_{i+1},b)\}$ is partial elementary.
Then $\cup_i f_i$ extends to an elementary map $f:K\rightarrow M^{\mathcal F}$.

For the second part, let $M\vDash T$ be the unique separable extremal model of $T$ (hence compact).
Since $E_n(T)$ is closed, all extremal models of $T$ are CL-equivalent (hence isomorphic) to $M$.
Therefore, $M$ is the unique extremal model of $T$.
\end{proof}
\vspace{1mm}

\begin{proposition}
Let $L$ be countable and $\mathbb T$ be a complete first order theory which can be axiomatized by linear axioms,
i.e. there is a linear theory $T$ with the same first order models as $\mathbb T$. Then, the following hold:

(i) If every $E_n(T)$ is closed, then $T_e=\mathbb{T}$.

(ii) If $\mathbb T$ is $\aleph_1$-categorical, then every $E_n(T)$ is closed.
\end{proposition}
\begin{proof}
(i) Let $K\vDash T$ be separable and extremal. Then, $K$ is first order and $T_e$ is the CL-theory of $K$.
By the assumption, $K\vDash\mathbb{T}$. So, $T_e=\mathbb{T}$. Let $M\vDash T$ be first order.
Then, $M\vDash\mathbb{T}=T_e$. So, $M$ is extremal.

(ii) Every extremal model of $T$ is first order, hence a model of $\mathbb T$.
On the other hand, by \ref{extremal4}, $T$ has extremal models of arbitrarily large cardinalities.
So, in fact, every model of $\mathbb T$ is extremal. Therefore, extremal models are axiomatized by $\mathbb T$.
\end{proof}

It is not hard to see that DLO, DAG, ODAG, $\mathrm{ACF}_p$, RCF (see \cite{Marker}) and many other
first order theories can be axiomatized by linear axioms with integer coefficients.
\vspace{2mm}

A theory $T$ has quantifier-elimination if for every formula $\phi$ there is a
quantifier-free formula $\psi$ such that $T\vDash\phi=\psi$.
We call a tuple $\a\in M\vDash T$ \emph{first order} if for every atomic $\theta(\x)$
one has that $\theta^M(\a)=0$ or $\theta^M(\a)=1$.

\begin{proposition}
Assume $L$ is countable and $T$ has quantifier-elimination. Then, every first order tuple is exposed.
In particular, if $T$ has a first order model $K$, then $T_e$ exists and coincides with the first order theory of $K$.
\end{proposition}
\begin{proof}
Let $\a\in M\vDash T$ be first order and $p(\x)=tp(\a)$.
Let $\theta_1,\theta_2,...$ be an enumeration of all atomic formulas with free variables $\x$.
Then $p$ is uniquely determined by the values $p(\theta_i)$.
Let $\hat\theta_i=\theta_i$ if $p(\theta_i)=1$ and $\hat\theta_i=1-\theta_i$ if $p(\theta_i)=0$.
Let $$\phi(\x)=\sum_{1\leqslant i} 2^{-i}\ \hat\theta_i(\x).$$
Although $\phi$ is not a finitary formula, $\phi^M(\a)$ and $p(\phi)$ are meaningful.
Furthermore, the (extended) condition $1\leqslant\phi(\x)$ exposes $p$,
i.e. realizations of $p$ are exactly the solutions of this condition
(or that, $p$ is the unique solution of $1\leqslant p(\phi)$).
In particular, every first order $M\vDash T$ is extremal. Conversely, by Proposition \ref{ext-first},
every extremal $M\vDash T$ is first order.
So, extremal models are exactly the first order models of $T$.
Since being a first order is CL-expressible, we conclude that $T_e$ exists and coincides with
the first order theory of any first order $K\vDash T$.
\end{proof}

\begin{proposition}
Let $L$ be countable and $M\vDash T$ be infinite and first order.
Then, every $E_n(T)$ is finite if and only if some extremal $K\vDash T$ has an
$\aleph_0$-categorical first order theory (in which case, $T_e$ is the first order theory of $K$).
\end{proposition}

\begin{proof}
Assume an extremal $K\vDash T$ has an $\aleph_0$-categorical theory.
Then, $K$ is oligomorphic, i.e. there are only finitely many orbits of automorphisms in any $K^n$.
So, by Proposition \ref{dense-realization}, every $E_n(T)$ is finite.
Conversely, assume every $E_n(T)$ is finite. We may replace $M$ with $M^{\mathcal F}$
where $\mathcal F$ is a countably incomplete ultrafilter on $\Nn$.
Let $K$ be a separable extremal (and hence first order) model of $T$.
We may further assume that $K\preccurlyeq M^{\mathcal F}$.
We show that $K$ has an $\aleph_0$-categorical first order theory.
Let $\a\in K$. Suppose $\b_1,\b_2,...\in M^{\mathcal F}$ have distinct extreme types over $\a$.
Then, $\a\b_1, \a\b_2,...$ have distinct extreme types which is impossible as every $E_n(T)$ is finite.
So, there are only finitely many extreme $n$-types over $\a$ and all these types must be realized in $K$
by Proposition \ref{dense-realization}.
Let $\a,\b\in K$,\ \ $\a\equiv\b$ and $c\in K$ be given. Then $tp(\a,c)=p(\x,y)$ is extreme.
Since $\b$ realizes $p|_{\x}$, $p(\b,y)$ is satisfiable and extreme by Lemma \ref{facetoface}.
So, it is realized by say $e\in K$ and hence $\a c\equiv\b e$.
An easy back and fourth construction shows that for $\a,\b\in K$,
if $\a\equiv\b$, there is an automorphism $f$ of $K$ such that $f(\a)=\b$.
Since for each $n$ there are only finitely many extreme $n$-types, $K$ is oligomorphic.
Therefore, it has an $\aleph_0$-categorical first order theory.
\end{proof}

Similar to the CL case, we call a complete linear theory $\kappa$-categorical if
it has a unique model of density character $\kappa$ up to isomorphism.
Recall that a compact convex set which is not a singleton, must have a non-extreme point.

\begin{corollary}
Let $T$ be a complete theory in a countable language $L$. Then $T$ is not $\aleph_0$-categorical.
\end{corollary}
\begin{proof}
There is nothing to prove if $T$ is trivial.
First, assume $T$ has no finite models. Let $a,b\in M\vDash T$ be distinct.
Then $tp(a,a)$ and $tp(a,b)$ are distinct. So, $S_2(T)$ has a non-extreme type say $p(x,y)$.
There is an infinite separable model of $T$ which omits $p$ and an infinite separable one which realizes $p$.
These models are not isomorphic.
Now, assume $T$ has a finite model $M$ (of cardinality $\kappa\geqslant2$).
Let $\mu$ be an arbitrary ultracharge on $\Nn$ with the property that $\mu(\{n\})>0$ for every $n$.
Then $M^{\mu}$ is a compact infinite separable model of $T$.
On the other hand, $T$ has a non-compact model say $N$ by the linear variant of upward L\"{o}wenheim-Skolem theorem.
Then, using the CL variant of the downward L\"{o}wenheim-Skolem theorem, we obtain a non-compact
separable elementary submodel $N_0\preccurlyeq N$. Now, $M^\mu$ and $N_0$ are non-isomorphic models of $T$.
\end{proof}

\begin{proposition}
If $T$ is ample, it is not $\kappa$-categorical for any $\kappa\geqslant\aleph_1$.
\end{proposition}
\begin{proof}
By Proposition \ref{extremal4} and downward L\"{o}wenheim-Skolem theorem,
there is an extremal model $M$ of $T$ with density character $\kappa$.
There is also a model with density character $\kappa$ which realizes some non-extreme type.
\end{proof}

\begin{proposition}
Assume $T$ is $\kappa$-categorical where $\kappa\geqslant|L|+\aleph_1$.
Then, all non-compact models of $T$ are $\mathrm{CL}$-equivalent.
If $T$ has no compact model, it is CL-complete.
\end{proposition}
\begin{proof}
Let $M,N\vDash T$ be noncompact. Using the full continuous variant of
L\"{o}wenheim-Skolem theorems, one can find $M_0\equiv_{\mathrm{CL}}M$
and $N_0\equiv_{\mathrm{CL}}N$ such that $M_0$ and $N_0$ have both density character $\kappa$.
Then, $M_0\equiv N_0$ and hence $M_0\simeq N_0$. Therefore, $M\equiv_{\mathrm{CL}}N$.
\end{proof}
\bigskip

We end the paper with a question concerning first order models of $T$.
A positive answer implies that for any first order $M,N$, \ $M\equiv N$ implies that $M\equiv_{\mathrm{CL}}N$.
\bigskip

\noindent{\sc Question}: For each complete theory $T$, every (possible) first order $M\vDash T$ is extremal.

\end{document}